\documentclass[12pt]{amsart}

\newtheorem{theorem}{Theorem}

\newtheorem{lemma}[theorem]{Lemma}

\theoremstyle{definition}

\theoremstyle{parrafo}

\begin{document}

\title[]{A refinement of the inequality between arithmetic and geometric means}

\author{J. M. Aldaz}

\address{Departamento de Matem\'aticas,
Universidad  Aut\'onoma de Madrid, Cantoblanco 28049, Madrid, Spain.}
\email{jesus.munarriz@uam.es}


\thanks{The author was partially supported by Grant MTM2006-13000-C03-03 of the
D.G.I. of Spain}









\maketitle


\markboth{J. M. Aldaz}{AM-GM}

  In this note we present
a refinement of the AM-GM inequality, and then we estimate in a special
case the typical size of the improvement.

\begin{theorem}\label{AMGMr}  For $i=1,\dots, n$, let $x_i\ge 0$,
suppose that some $x_i > 0$, and let
$\alpha_i > 0$ satisfy $\sum_{i=1}^n \alpha_i = 1$. Then
\begin{equation}\label{refAMGM}
\exp\left(2  \left(1 - \frac{\sum_{i=1}^n \alpha_i x_i^{1/2}}
{\left(\sum_{i=1}^n \alpha_i x_i\right)^{1/2}}\right) \right)
\prod_{i=1}^n x_i^{\alpha_i}
\le
\sum_{i=1}^n \alpha_i x_i.
\end{equation}
\end{theorem}
The hypotheses  $n\ge 2$,
  $x_i\ge 0$ for $i=1,\dots, n$ and
$\alpha_i > 0$, $\sum_{i=1}^n \alpha_i = 1$ will be maintained
throughout  this note without further mention.

\begin{proof}
The following refinement of
a standard consequence of H\"older's inequality on a probability
space appears in \cite[Theorem 3.1]{Al}:
Let $0 < r < s/2 < \infty$, and let
$0\le f\in L^s$
satisfy
 $\|f\|_s > 0$. Then
\begin{equation}\label{refint}
  \|f\|_r
 \le  \|f\|_s \left[1 - \frac{2 r}{s}
\left(1 - \frac{\|f^{s/2}\|_1}{\|f^{s/2}\|_2}\right) \right]^{1/r}.
\end{equation}

From this, a well known argument yields the refinement of the
AM-GM indicated in the theorem.
Let
$f:[0,\infty) \to [0,\infty)$ be the identity $f(x) = x$, and let
$\mu := \sum_{i=1}^n \alpha_i \delta_{x_i}$, where $\delta_{x_i}$ denotes
the Dirac measure supported on $\{x_i\}$.
Writing
$$
c:= 1 - \frac{\sum_{i=1}^n \alpha_i x_i^{1/2}}
{\left(\sum_{i=1}^n \alpha_i x_i\right)^{1/2}}
$$
and setting $s = 1$ in (\ref{refint}), we get
\begin{equation}\label{r}
\left(\sum_{i=1}^n \alpha_i x_i^r\right)^{1/r}
\le \left(\sum_{i=1}^n \alpha_i x_i\right) \left( 1 -2r c\right)^{1/r}.
\end{equation}
Taking the limit as $r\downarrow 0$, it follows
(for instance, by L'H\^opital's rule) that
$$
\prod_{i=1}^n x_i^{\alpha_i} \le
\left(\sum_{i=1}^n \alpha_i x_i\right) e^{-2c}.
$$
\end{proof}

Regarding the meaning of (\ref{refint}),
it essentially
says that, for fixed $r$ and $s$, if the variance of $f^{s/2}/\|f^{s/2}\|_2$ is large,
then $\|f\|_r
<<  \|f\|_s$. For simplicity, set $s=1$, which is the case we  used. To see that
$1 - \int f^{1/2}/\left(\int f\right)^{1/2}$ is a measure of the
dispersion of $f^{1/2}$ about its mean value, and in fact,
comparable to the variance $\operatorname{Var}\left(f^{1/2}/\|f^{1/2}\|_2\right)$ of its normalization in $L^2$, observe first that
$\int f^{1/2}/\left(\int f\right)^{1/2} \le 1$, since $\operatorname{Var}\left(f^{1/2}\right) \ge 0$.
Now, for all
$t\in [0,1]$
\begin{equation}\label{eleminequ}
 2^{-1} (1 - t^2) = 2^{-1} (1 + t) (1-t) \le  1 - t \le  1 - t^2,
\end{equation}
so, setting $t = \|f^{1/2}\|_1/\|f^{1/2}\|_2$, we obtain\begin{equation}\label{var}
\frac{1}{2}
\operatorname{Var}\left(\frac{f^{1/2}}{\|f^{1/2}\|_2}\right) \le
1 - \frac{\|f^{1/2}\|_1}{\|f^{1/2}\|_2}
  \le
\operatorname{Var} \left(\frac{f^{1/2}}{\|f^{1/2}\|_2}\right).
 \end{equation}

Using  (\ref{var}), we see that (\ref{refAMGM}) entails the following
inequality:
\begin{equation}\label{VarAMGM}
\exp\left(\operatorname{Var}\left(\frac{x^{1/2}}{\|x^{1/2}\|_2}\right) \right) \prod_{i=1}^n x_i^{\alpha_i} \le
\sum_{i=1}^n \alpha_i x_i.
\end{equation}
 Thus, (\ref{VarAMGM})
gives a quantitative bound of the deviation from equality, in terms of the variance of   $x^{1/2}/\|x^{1/2}\|_2$; if the variance is large, so is
the difference between the AM and the GM.

\vskip .2cm

Next we ask ourselves how ``efficient"  the refinement in (\ref{refAMGM})
is.
 We
 study its average performance in the classical equal weights case, modified by the
change of variables $x_i = y_i^2$. The AM-GM inequality bounds the
GM-AM ratio by 1, always, while for $n >>1$ and after the said change
of variables, inequality (\ref{refAMGM}) gives a ``typical"
upper bound smaller than $0.82$ (with probability at least $1 - 1/n$).
We prove this next.

Let $\alpha_i = 1/n$
for $i=1,\dots, n$,
write $x_i = y_i^2$, and  set $y = (y_1,\dots,y_n)$,
where $y_i$ is now allowed to take negative values. In terms of the GM-AM ratio
(\ref{refAMGM}) becomes
\begin{equation}\label{r2}
\frac{\prod_{i=1}^n |y_i|^{1/n}}{\sqrt{\frac{1}{n} \sum_{i=1}^n y^2_i}}
\le
\exp\left(\frac{\frac{1}{n} \sum_{i=1}^n   |y_i|}
{\left(\frac{1}{n} \sum_{i=1}^n  y^2_i\right)^{1/2}}  -1\right).
\end{equation}
The $\ell_1^n$ and $\ell_2^n$ norms of the vector $y\in\mathbb{R}^n$
are given by $\|y\|_1 := \sum_{i=1}^n   |y_i|$ and
$\|y\|_2 := \sum_{i=1}^n   y_i^2$ respectively. Since both
sides of (\ref{r2}) are positive homogeneous functions of
degree zero (so replacing $y$ by $y/\|y\|_2$ does not change anything)
we may assume that $\|y\|_2 = 1$. Probability statements then mean
that $y$ is chosen at random (i.e., uniformly) from the euclidean unit sphere. Setting $\|y\|_2 = 1$,  inequality (\ref{r2}) becomes
\begin{equation}\label{r2simp}
\sqrt{n} \prod_{i=1}^n |y_i|^{1/n}
\le
\exp\left(n^{- 1/2} \|y\|_1 -1 \right).
\end{equation}

Denote by $P^{n-1}$ the normalized area, or Haar measure, on
 the euclidean unit sphere
$\mathbb{S}^{n-1}_2 = \{\|y\|_2=1\}\subset\mathbb{R}^n$.

\begin{theorem} \label{typical}
For all $n$ sufficiently  high and with probability  at least $1-1/n$
on $\mathbb{S}^{n-1}_2$, we have
\begin{equation}\label{lowering}
\exp\left(n^{- 1/2} \|y\|_1 -1 \right)
< 0.82.
\end{equation}
\end{theorem}

The proof consists in computing the expectation of
 $n^{- 1/2}\|y\|_1$
over
$\mathbb{S}^{n-1}_2$ (a standard calculation) and then using the
known fact
that typically $n^{- 1/2}\|y\|_1$ is very close to its mean, provided
$n$ is large enough. Details follow.

 Recall that the area of $\mathbb{S}^{n-1}_2$ is
$|\mathbb{S}^{n-1}_2|= 2\pi^{n/2}/\Gamma(n/2)$, and
that $\|y\|_2 \le \|y\|_1\le \sqrt{n} \|y\|_2$. It turns out that
the average of $\|\cdot\|_1$ over $\mathbb{S}^{n-1}_2$ is closer to $\sqrt{n}$ than to 1.

\begin{lemma}\label{lemma1}   The expectation  of $\|\cdot\|_1$ over $\mathbb{S}_2^{n-1}$ is given by
\begin{equation}\label{mean1}
E \left( \|y\|_1\right) :=
\int_{\mathbb{S}_2^{n-1}}  \|y\|_1 d P^{n-1}(y)
=
\frac{n \Gamma\left(\frac{n}{2} \right)}
{\pi^{1/2}\Gamma\left(\frac{n + 1}{2}\right)}.
\end{equation}
Thus,
\begin{equation}\label{downbupb}
\sqrt{\frac{2}{\pi}}
\le
E \left(\frac{ \|y\|_1}{\sqrt{n}}\right)
\le
\left(\frac{n}{n - 1}\right)^{1/2} \sqrt{\frac{2}{\pi}}.
\end{equation}
\end{lemma}

 \begin{proof}
 We integrate the left hand side of (\ref{integr}) below
in two ways, first in   polar coordinates
and then as a product, via Fubini's Theorem.  Using $|\mathbb{S}^{n-1}_2|= 2\pi^{n/2}/\Gamma(n/2)$ and the fact that $\|\cdot\|_1$ is a positive homogeneous function
of degree one, we get
\begin{equation}\label{integr}
\int_{\mathbb{R}^n}  \|y\|_1
e^{ -  \|y\|_2^2 } dy
 =
\int_0^\infty t^{n} e^{- t^2}  dt
\left(\frac{2\pi^{n/2}}{\Gamma\left(\frac{n}{2}\right)}\right)
\int_{\{\|y\|_2 = 1\}} \|y\|_1
 dP^{n-1}(y).
\end{equation}
Given $y = (y_1,\dots,y_n)\in \mathbb{R}^{n}$, we denote by
$\hat{y}_i\in \mathbb{R}^{n-1}$ the vector
obtained from $y$ by deleting the $i$-th coordinate:
$\hat{y}_i=(y_1, \dots, y_{i-1}, y_{i+1},\dots, y_n)$.
Now the first two integrals in (\ref{integr}) can either be
computed or expressed in terms of the Gamma function:
\begin{equation}\label{int1}
\int_0^\infty t^{n} e^{- t^2}  dt = \frac12\Gamma\left(\frac{n + 1}{2}\right),
\end{equation}
and
\begin{equation}\label{int2}
\int_{\mathbb{R}^n}  \left(\sum_{i=1}^n |y_i|  \right)
\exp\left( - \sum_{i=1}^n y^2_i \right)  dy
 =  \sum_{i=1}^n 2 \int_0^\infty y_i e^{- y_i^2}  dy_i
 \int_{\mathbb{R}^{n-1}} \exp\left( - \|\hat{y}_i\|_2^2 \right) d \hat{y}_i
=  n \pi^{(n-1)/2}.
 \end{equation}
Putting together (\ref{integr}), (\ref{int1}) and (\ref{int2}), and solving
for the expectation, we get
\begin{equation}\label{expect}
E \left( \|y\|_1\right)
=
\frac{n \Gamma\left(\frac{n}{2} \right)}
{\pi^{1/2}\Gamma\left(\frac{n + 1}{2}\right)}.
 \end{equation}
Now (\ref{downbupb}) follows from (\ref{expect}) by using the following known
 and elementary
estimate (cf. Exercise 5, pg. 216 of \cite{Web}; the result is an
immediate consequence of the log-convexity of the $\Gamma$ function):\begin{equation}\label{ratio}
\sqrt{\frac{2}{n}}
\le
\frac{\Gamma \left(\frac{n}{2}\right)}{\Gamma \left(\frac{n + 1}{2}\right)}
\le
\sqrt{\frac{2}{n-1}}.
\end{equation}
\end{proof}

An expression similar to (\ref{downbupb}) can also be obtained
from (\ref{expect}) by using the very well known asymptotic expansion
\begin{equation*}
\Gamma (z) = e^{-z} z^{z - 1/2} \sqrt{2\pi} \left(1 + \frac{1}{12 z}
+ O(z^{-2})\right).
\end{equation*}

Given a real valued random variable $f$ on a
probability space $(X, \mu)$, a median $M_f$ of $f$
is a constant such that $\mu\{f\ge M_f\} \ge 1/2$ and
$\mu\{f\le M_f\} \ge 1/2$. It is a well known fact   that ``reasonable functions" on $\mathbb{S}_2^{n-1}$,
when observed at the right scale, exhibit the concentration
of measure phenomenon. That is, they are almost constant over large
portions of the sphere,
taking values very close to their medians.
In the particular case of
$\|\cdot\|_1$, the right scale means dividing by
$\sqrt{n}$, which is precisely what we have in the
right hand side of (\ref{r2simp}). We use the following facts (they
can be found in \cite{MiSch}, within
the proof of the Lemma in pg. 19):

1) $\left|E \left( \|y\|_1\right) - M_{\|y\|_1}\right|\le \pi/2$.

2) $P^{n-1}\{\left| M_{\|y\|_1} - \|y\|_1 \right| > t \}
\le
\sqrt{\frac{\pi}{2}} e^{-t^2/2}.$

\vskip .2 cm

{\em Proof of Theorem \ref{typical}.} Let $t= \sqrt{\log ((\pi/2) n^2)}$.   By 1),
$$
\{\left| E\left(\|y\|_1\right) - \|y\|_1 \right| > t  + \pi /2\}
\subset
\{\left|\|y\|_1 - M_{\|y\|_1}\right| > t\}.
$$
By the preceding inclusion and 2),
$$P^{n-1}\{ \|y\|_1 > t + \pi/2 + E\left(\|y\|_1\right) \}
 \le
P^{n-1}\{ |E\left(\|y\|_1\right) - \|y\|_1 |  > t + \pi/2\}
$$
$$
\le
P^{n-1}\{\left|\|y\|_1 - M_{\|y\|_1}\right| > t\}\le 1/n.
 $$
It follows from (\ref{downbupb}) and the previous bound that  for all $n$ sufficiently  high,
$$
\|y\|_1 \le t + \frac{\pi}{2} + \left(\frac{n}{n - 1}\right)^{1/2} \sqrt{\frac{2 n}{\pi}},
$$
with probability at least $1-1/n$.
Since
$n^{- 1/2} \left(t + \pi/2 \right) = O(n^{- 1/4})$
and $\left(n/(n - 1)\right)^{1/2}= 1 + O(n^{- 1})$,
again for all $n$ sufficiently  high and with probability at least $1-1/n$
we have
\begin{equation}
\exp\left(n^{- 1/2} \|y\|_1 -1 \right)
\le
\exp\left(\sqrt{\frac{2}{\pi}}\left(1 + O\left(\frac{1}{n^{1/4}}\right)\right) -1 \right)
< 0.82.
\end{equation}
\qed

\vskip .2 cm

In some  unlikely cases  (that is, with low
 $P^{n-1}$-probability) the refined AM-GM inequality  (\ref{refAMGM})
performs badly. Suppose $n >> 1$, and
let  $0 < y_1= \dots = y_n$. Then both sides of
(\ref{r2}) equal 1. Letting $y_1\downarrow 0$,   the left hand side drops to zero, while the right hand side remains
essentially unchanged.

Finally, given any correction factor in a refinement of the AM-GM inequality, how far down could it
go on some large set? Not lower than $1/2$. Actually,
it is possible to be more precise: In \cite{GluMi} E. Gluskin and V. Milman show that
$0.394 <n^{1/2} \prod_{i=1}^n |y_i|^{1/n}$ asymptotically in $n$, with
probability approaching $1$, and their method can be used to prove  that
$n^{1/2} \prod_{i=1}^n |y_i|^{1/n}$
concentrates around the value $\sqrt{2} \exp\left[\Gamma^\prime(1/2)/(2\Gamma(1/2))\right]\approx 0.529$
(cf. \cite[Theorem 2.8]{Al2}).

\end{document}